\documentclass[12pt,reqno]{amsart}
\usepackage{amssymb, amsmath,amsthm}
\usepackage{color}
\usepackage{cite}
\newtheorem{thm}{Theorem}

\numberwithin{defn}{section}
\numberwithin{thm}{section}
\numberwithin{Lemma}{section}
\numberwithin{Corollary}{section}
\numberwithin{Example}{section}
\numberwithin{subsection}{section}
\numberwithin{Remark}{section}
\numberwithin{equation}{section}
\numberwithin{ppn}{section}
\begin{document}
\title
[Extended Local Convergence for Seventh order ... ]
{Extended Local Convergence for Seventh order method with $\psi$-continuity condition in Banach Spaces} 
\author{Akanksha Saxena, J. P. Jaiswal, K. R. Pardasani}
\date{}
\maketitle

\textbf{Abstract.} 
In this article, the local convergence analysis of the multi-step seventh order method is presented for solving nonlinear equations. The point worth noting in our paper is that our analysis requires a weak hypothesis where the Fr\'echet derivative of the nonlinear operator satisfies the $\psi$-continuity condition and extends the applicability of the computation when both Lipschitz and H\"{o}lder conditions fail. The convergence in this study is shown under the hypotheses on the first order derivative without involving derivatives of the higher-order.
To find a subset of the original convergence domain, a strategy is devised. As a result, the new Lipschitz constants are at least as tight as the old ones, allowing for a more precise convergence analysis in the local convergence case. Some numerical examples are provided to show the performance of the method presented in this contribution over some existing schemes.
\\\\
\textbf{Mathematics Subject Classification (MSC2020).} 
65H10, 65J15, 65G99, 47J25.


\section{\bf Introduction}

Numerical analysis is a wide-ranging subject having close connections with mathematics, computer science, engineering and the applied sciences. One of the most elemental and primal problem in numerical analysis, when mathematically modeled lead to integral equations, boundary value problems and differential problems, concerns with converting to nonlinear equations having the form of
\begin{equation}\label{eqn:11}
T(x)=0,
\end{equation}
where $T$ is defined on a convex open subset $D$ of a Banach space $X$ with values in a Banach space $Y$. Many problems in different fields of computational science, engineering such as radiative transfer theory, optimization etc. can be brought in a form like $(\ref{eqn:11})$ using mathematical modeling. Analytical methods of solving such type of problems are very scarce or almost non existent. Therefore, many researchers only rely on iterative method and they have proposed a plethora of iterative methods. In various articles, many authors have studied the local convergence analysis using Taylor's series but don't obtain the radii of convergence ball for the solution which can be seen in the refs. (  \cite{Behl},\cite{Sharma}). This approach has been extended for the iterative method in Banach spaces for obtaining better theoretical results without following Taylor's series approach. In this way, there is no need to use the higher-order derivatives to show the convergence of the scheme. These types of techniques are discussed by many authors, for example, one can go through the refs. (\cite{Rall}-\cite{Traub}). The practice of numerical functional analysis for finding such solutions are widely and substantially connected to Newton-like methods which is defined as follows
\begin{eqnarray}
x_{n+1}=x_n-[T'(x_n)]^{-1}T(x_n), \ \ \ n\ge 0,
\end{eqnarray}
is frequently used by various researchers as it has quadratic convergence (can be seen in the reference \cite{Traub}) and only one evaluation of the Jacobian of $T$ is needed at each step. The other properties of Newton's method are established in the reference book \cite{Kantorovich}. Moreover, in some applications involving stiff systems, high-order methods are useful. Therefore, it is important to study high-order methods.

Obtaining the radius of the convergence ball, as well as devising a theory to extend the convergence region, are both important issues. The convergence domain is critical for the steady behaviour of an iterative method from a numerical stand point. The convergence analysis of iterative procedures, particularly  the local analysis is based on the information around a solution, to find estimates of the radii of the convergence balls. Plenty of studies have been conducted on the local and semilocal convergence analysis of Newton-like techniques. Many iterative methods of increasing order of convergence, such as third order (\cite{Argyros2}-\cite{Argyros1}), fourth order (\cite{Argyros3}) and fifth order (\cite{Cordero}-\cite{Martinez}) etc., have been developed in recent decades and have demonstrated their efficiency in numerical terms.

In particular, Sharma and Gupta \cite{Janak} constructed three-step method of order five, mentioned as follows:
{\small
\begin{eqnarray}\label{eqn:12}
y_n&=&x_n-\frac{1}{2}\Gamma_nT(x_n),\nonumber\\
z_n&=&x_n-[T'(y_n)]^{-1}T(x_n),\nonumber\\
x_{n+1}&=&z_n-[2[T'(y_n)]^{-1}-\Gamma_n]T(z_n),
\end{eqnarray}}
 where, $\Gamma_n=[T'(x_n)]^{-1}$.
The local convergence of the above multi-step Homeier's-like method has been studied by Panday and Jaiswal \cite{Bhavna} with the help of Lipschitz and H$\ddot{o}$lder continuity conditions. 
{\small
\begin{eqnarray}\label{eqn:13}
y_n&=&x_n-\frac{1}{2}\Gamma_nT(x_n),\nonumber\\
z_n^{(1)}&=&x_n-[T'(y_n)]^{-1}T(x_n),\nonumber\\
z_{n}^{(2)}&=&z_n^{(1)}-[2[T'(y_n)]^{-1}-\Gamma_n]T(z_n^{(1)}),\nonumber\\
x_{n+1}&=&z_n^{(2)}-[2[T'(y_n)]^{-1}-\Gamma_n]T(z_n^{(2)}).
\end{eqnarray}}
This method requires the evaluation of three function, two first order derivatives and two matrix inversions per iteration. In this article, we have weaken the continuity condition and then analyze its local convergence. The motivation for writing this paper is the extension of the applicability of method $(\ref{eqn:13})$ by using the novelty that includes the extension of the convergence domain, that can be illustrated by creating a subset $D$ that includes the iterates. Nevertheless, because the Lipschitz-like parameters (or functions) in this set $D$ are at least as tight as the originals, the convergence is finer.

Moreover, numerous problems are there for which Lipschitz as well as H\"{o}lder condition both fail. As an motivational illustration, consider the nonlinear integral equation of the mixed Hammerstein-type given by \cite{Martinez}.
\begin{eqnarray*}
F[x(s)]=x(s)-5\int_{0}^1stx(t)^3dt,
\end{eqnarray*}
with $x(s)$ in $C[0,1]$. The first derivative of $F$ is
\begin{eqnarray*}
F'[x(s)]v(s)=v(s)-15\int_{0}^1stx(t)^2dt.
\end{eqnarray*}
It is clear that both Lipschitz and H\"{o}lder condition do not hold for this problem. Thereby, we expand the applicability of method $(\ref{eqn:13})$ by using hypothesis only on the first-order derivative of the function $T$ and generalized Lipschitz continuity conditions. In this manuscript, we address many concerns by providing the radius of the convergence ball, computable error bounds and the uniqueness of the solution of the result on using the weaker continuity condition.

The outline of this paper is as follows: Section $2$ deals with the local convergence results for the method $(\ref{eqn:13})$, obtaining a ball of convergence followed by its uniqueness. The numerical examples appear before the concluding Section. At last, we discuss the obtained results and discussions.

\section{\bf Local convergence analysis}
The local convergence analysis of the method $(\ref{eqn:13})$ is centered on some parameters and scalar functions. Let $\psi_0,\ \psi$ be a  non-decreasing continuous functions defined on the interval $[0,+\infty)$ with values in $[0,+\infty)$ satisfying $\psi_0(0)=0$. Define parameters $\rho_0$ by
\begin{equation}\label{eqn:31}
\rho_0=sup\{t\ge 0:\psi_0(t)<1\}.
\end{equation}
Let also $\psi :[0,r_0)\rightarrow [0,+\infty)$ be continuous and non-decreasing functions so that $\psi(0)=0$. Define the functions $\eta_1,\eta_2,\eta_3,\eta_4,p,H_1,H_2,H_3$ and $H_4$ on interval $[0,\rho_0)$ by
{\tiny
\begin{eqnarray}\label{eqn:32}
\eta_4(a)&=&{\bigg[}\frac{\int_0^1\psi((1-t)\eta_3(a)a)dt\eta_3(a)}{(1-\psi_0(\eta_3(a)a))}\nonumber\\
&&+\frac{[\psi(\eta_1(a).a)+\psi(\eta_3(a)a)]}{1-\psi_0(\eta_3(a)a)}\times
\left(\frac{1}{1-p(a)}\bigg[\int_0^1{\bigg(}\psi_0(t\eta_3(a)a){+1\bigg)}dt\bigg]\eta_3(a)\right)\nonumber\\
&&+\left(\frac{1}{1-\psi_0(a)}.\frac{L[\psi(\eta_1(a).a)+\psi(a)]}{(1-p(a)}\times \bigg[1+\frac{L_0}{2}\eta_3(a)a\bigg]\right){\bigg]}\eta_3(a),
\end{eqnarray}}
where
{\small
\begin{eqnarray}\label{eqn:33}
\eta_1(a)=\frac{1}{1-\psi_0(a)}\bigg[\frac{1}{2}\int_0^1{\bigg(}\psi_0(ta){+1\bigg)}dt+\int_0^1\psi((1-t)a)dt\bigg],
\end{eqnarray}}
{\small
\begin{eqnarray}\label{eqn:34}
\eta_2(a)=\frac{1}{1-\psi_0(a)}\bigg[\int_0^1\psi((1-t)a)dt+\frac{[\psi(\eta_1(a).a)+\psi(a)][\int_0^1{\bigg(}\psi_0(ta){+1\bigg)}dt]}{1-p(a)}\bigg],\nonumber\\
\end{eqnarray}}
{\tiny
\begin{eqnarray}\label{eqn:32a}
\eta_3(a)&=&{\bigg[}\frac{\int_0^1\psi((1-t)\eta_2(a)a)dt\eta_2(a)}{(1-\psi_0(\eta_2(a)a))}\nonumber\\
&&+\frac{[\psi(\eta_1(a).a)+\psi(\eta_2(a)a)]}{1-\psi_0(\eta_2(a)a)}\times
\left(\frac{1}{1-p(a)}\bigg[\int_0^1{\bigg(}\psi_0(t\eta_2(a)a){+1\bigg)}dt\bigg]\eta_2(a)\right)\nonumber\\
&&+\left(\frac{1}{1-\psi_0(a)}.\frac{L[\psi(\eta_1(a).a)+\psi(a)]}{(1-p(a)}\times \bigg[1+\frac{L_0}{2}\eta_2(a)a\bigg]\right){\bigg]}\eta_2(a),
\end{eqnarray}}
and
\begin{eqnarray}\label{eqn:35}
p(a)=\psi_0(\eta_1(a).a).
\end{eqnarray}
Let
{\small
\begin{eqnarray}\label{eqn:36}
H_1(a)=\eta_1(a)-1,\ \ \ \ H_2(a)=\eta_2(a)-1,\\
H_3(a)=\eta_3(a)-1,\ \ \ \ H_4(a)=\eta_4(a)-1.
\end{eqnarray}}
We have that $H_1(0)=H_2(0)=H_3(0)=H_4(0)<0$.\\
Suppose $H_1(a)\rightarrow+\infty$ or a positive constant and \ $H_2(a)\rightarrow+\infty$ or a positive constant as $t\rightarrow \rho_0^-$.
Similarly, $H_3(a)\rightarrow+\infty$ or a positive constant and $H_4(a)\rightarrow+\infty$ or a positive constant as $t\rightarrow \overline{\rho}_0^-$,\\
where
\begin{equation}\label{eqn:31a}
\overline{\rho}_0=max\{a\in[0,\rho_0]:\psi_0(\eta_1(a).a)<1\}.
\end{equation}
It then follows from the intermediate value theorem that functions $H_i,\ i=1,2,3,4$ have zeros in the interval $(0,\rho_0)$. Define the radius of convergence $\rho$ by
\begin{eqnarray}\label{eqn:37}
\rho=min\{\rho_i\}, \ i=1,2,3,4;
\end{eqnarray}
where $\rho_i's$ denote the smallest solution of functions $H_i's$.
Then, we have that for each $a\in [0, \rho)$
\begin{eqnarray}\label{eqn:38}
0\leq \eta_i(a)<1,\\
0\leq \psi_0(a)<1,\\
0\leq \psi_0(\eta_1(a).a)<1.
\end{eqnarray}
Let $B(x^*,\rho),\ \overline{B(x^*,\rho)}$ stand, respectively for the open and closed ball in $X$ such that $x^*\in X$ and of radius $\rho>0$. Next, we present the local convergence analysis of method $(\ref{eqn:13})$ using the preceding notations and generalized Lipschitz-H$\ddot{o}$lder type conditions.

\begin{thm}\label{thm:31}
Let $T:D\subset X\rightarrow Y$ be a continuously first order Fr\'echet differentiable operator. Suppose that $x^*\in D$ and function $\psi_0:[0,+\infty)\rightarrow [0,+\infty)$ with $\psi_0(0)=0$, continuous and non-decreasing such that for each $x\in D$
\begin{eqnarray}\label{eqn:311}
T(x^*)=0,\ [T'(x^*)]^{-1}\in L(Y,X),
\end {eqnarray}
where, $L(X,Y)$ is the set of bounded linear operators from $X$ to $Y$,
\begin{eqnarray}\label{eqn:312}
\|[T'(x^*)]^{-1}(T'(x)-T'(x^*))\|\le \psi_0(\|x-x^*\|).
\end{eqnarray}
Moreover, suppose that there exists function $\psi:[0,+\infty)\rightarrow [0,+\infty)$ with $\psi(0)=0$, continuous and non-decreasing such that for each $x,y \ in \ D_0=D\cap B(x^*,\rho_0)$
\begin{eqnarray}\label{eqn:313}
\|[T'(x^*)]^{-1}(T'(x)-T'(y))\|&\le& \psi(\|x-y\|),\\
B(x^*,\rho)&\subseteq& D,
\end{eqnarray}
where $\rho_0$ and $\rho$ defined by equations $(\ref{eqn:31})$ and $(\ref{eqn:37})$, respectively. Then, the sequence $\{x_n\}$ generated by method $(\ref{eqn:13})$ for $x_0\in B(x^*,\rho)\backslash\{x^*\}$ is well defined in $B(x^*,\rho)$ remains in $B(x^*,\rho)$ for each $n=0,1,2,\cdots$ and converges to $x^*$. Moreover, the following estimates holds:
{\small
\begin{eqnarray}\label{eqn:316}
\|y_n-x^*\|\le \eta_1(\|x_n-x^*\|)\|x_n-x^*\|\le\|x_n-x^*\|<\rho,\\
\|z_n^{(1)}-x^*\|\le \eta_2(\|x_n-x^*\|)\|x_n-x^*\|\le\|x_n-x^*\|<\rho,\\
\|z_n^{(2)}-x^*\|\le \eta_3(\|x_n-x^*\|)\|x_n-x^*\|\le\|x_n-x^*\|<\rho,
\end{eqnarray}}
and
{\small
\begin{eqnarray}\label{eqn:318}
\|x_{n+1}-x^*\|\le \eta_4(\|x_n-x^*\|)\|x_n-x^*\|\le\|x_n-x^*\|<\rho,
\end{eqnarray}}
where the functions $\eta_i, i=1,2,3,4$ are defined by the expressions $(\ref{eqn:32})$ - $(\ref{eqn:32a})$. Furthermore, if there exists $\varrho\ge \rho$ such that 
\begin{equation}\label{eqn:319}
\int_0^1\psi_0(\theta \varrho)d\theta <1,
\end{equation}
then the point $x^*$ is the only solution of equation $T(x)=0$ in $D_1=D\cap \overline {B(x^*,\varrho)}$.
\end{thm}

\begin{proof}
We shall show by mathematical induction that sequence $\{x_n\}$ is well defined and converges to $x^*$. Using the hypotheses, $x_0\in B(x^*,\rho)\backslash\{x^*\}$, equation $(\ref{eqn:31})$ and inequality $(\ref{eqn:312})$, we have that
\begin{eqnarray}
\|[T'(x^*)]^{-1}(T'(x_0)-T'(x^*))\|\le\psi_0(\|x_0-x^*\|)\le\psi_0(\rho)<1.
\end{eqnarray}
It follows from the above and the Banach lemma on invertible operator \cite{Argyros} that $[T'(x_0)]^{-1}\in L(Y,X)$
 or $T'(x_0)$ is invertible and
\begin{eqnarray}
\|[T'(x_0)]^{-1}T'(x^*)\|\le\frac{1}{1-\psi_0(\|x_0-x^*\|)}.
\end{eqnarray}
Now, $y_0$ is well defined by the first sub-step of the scheme $(\ref{eqn:13})$ and for $n=0,$
{\small
\begin{eqnarray}\label{eqn:321}
y_0-x^*&=&x_0-x^*-\frac{1}{2}[[T'(x_0)]^{-1}T(x_0)]\nonumber\\
& =&\frac{1}{2}[[T'(x_0)]^{-1}T(x_0)]+[T'(x_0)]^{-1}[T'(x_0)(x_0-x^*)-T(x_0)+T(x^*)]||.\nonumber\\
\end{eqnarray}}
Expanding $T(x_0)$ along $x^*$ and taking the norm of the equation $(\ref{eqn:321})$, we get
{\small
\begin{eqnarray}\label{eqn:322}
\|y_0-x^*\|&\le& \bigg\|\frac{1}{2}[T'(x_0)]^{-1}T'(x^*)\bigg\|\bigg[\int_0^1\|[T'(x^*]^{-1}[T'(x*+t(x_0-x^*)]dt\|\|x_0-x^*\|\bigg]\nonumber\\
&&+\|[T'(x_0)]^{-1}T'(x^*)\|\int_0^1\|[T'(x^*)]^{-1}[T'(x_0)-T'(x^*+t(x_0-x^*))]dt\|\|x_0-x^*\|,\nonumber\\
&\le& \frac{1}{1-\psi_0(\|x_0-x^*\|)}\bigg[\frac{1}{2}\int_0^1{\bigg[}\psi_0(t\|x_0-x^*\|)+1{\bigg]} dt\nonumber\\
&+&\int_0^1\psi((1-t)\|x_0-x^*\|)dt\bigg]\|x_0-x^*\|\nonumber\\
&\le&\eta_1(\|x_0-x^*\|)\|x_0-x^*\|<\rho.
\end{eqnarray}}
From the inequalities $(\ref{eqn:312})$ and $(\ref{eqn:322})$, we have
{\small
\begin{eqnarray}
\|[T'(x^*)]^{-1}[T'(y_0)-T'(x^*)]\|&\le&\psi_0(\|y_0-x^*\|)\nonumber\\
&\le&\psi_0(g_1(\|x_0-x^*\|)\|x_0-x^*\|)\nonumber\\
&=&p(\|x_0-x^*\|)<1.
\end{eqnarray}}
Thus, by Banach lemma,
\begin{eqnarray}
\|[T'(y_0)]^{-1}T'(x^*)\|\le\frac{1}{1-p(\|x_0-x^*\|)}.
\end{eqnarray}
From the second sub-step of the method $(\ref{eqn:13})$, we have
\begin{eqnarray}\label{eqn:326}
z_0^{(1)}-x^*&=&x_0-x^*-[T'(y_0)^{-1}T(x_0)]\nonumber\\
&=&x_0-x^*-[T'(x_0)]^{-1}T(x_0)+[T'(x_0)]^{-1}[T'(y_0)-T'(x_0)]T'(y_0)^{-1}T(x_0).\nonumber\\
\end{eqnarray}
On taking norm of the equation $(\ref{eqn:326})$, we get
{\small
\begin{eqnarray}
\|z_0^{(1)}-x^*\|&\le&\|x_0-x^*-[T'(x_0)]^{-1}T(x_0)\|+\|[T'(x_0)]^{-1}T'(x^*)\|.\nonumber\\
&&\|[T'(x^*)]^{-1}[T'(y_0)-T'(x_0)]\|\|[T'(y_0)]^{-1}T'(x^*)\|\|[T'(x^*)]^{-1}T(x_0)\|\nonumber\\
&\le&\frac{1}{1-\psi_0(\|x_0-x^*\|)}\bigg[\int_0^1\psi((1-t)\|x_0-x^*\|)dt\nonumber\\
&+&\frac{[\psi(\|y_0-x^*\|)+\psi(\|x_0-x^*\|)][\int_0^1{\bigg(}\psi_0(t\|x_0-x^*\|){+1\bigg)}dt]}{1-p(\|x_0-x^*\|)}\bigg]\|x_0-x^*\|.\nonumber\\
\end{eqnarray}}
Thus, we get
\begin{eqnarray}
\|z_0^{(1)}-x^*\|\le \eta_2(\|x_0-x^*\|)\|x_0-x^*\|\le\|x_0-x^*\|<\rho.
\end{eqnarray}
From the next sub-step of the method $(\ref{eqn:13})$, we have
{\small
\begin{eqnarray}\label{eqn:327}
&&z_0^{(2)}-x^*=z_0^{(1)}-x^*-(2[T'(y_0)]^{-1}-[T'(x_0)]^{-1})T(z_0^{(1)})\nonumber\\
&&=(z_0^{(1)}-x^*-[T'(z_0^{(1)})]^{-1}T'(z_0^{(1)}))+[T'(z_0^{(1)})]^{(-1)}T'(x^*)T'(x^*)^{-1}[T'(y_0)-T'(z_0^{(1)})]\nonumber\\
&&.[T'(y_0)]^{(-1)}T'(x^*)[T'(x^*)]^{-1}T(z_0^{(1)})+[T'(x_0)]^{(-1)}T'(x^*)[T'(x^*)]^{-1}[T'(y_0)-T'(x_0)]\nonumber\\
&&.[T'(y_0)]^{(-1)}T'(x^*)[T'(x^*)]^{-1}T(z_0^{(1)}).
\end{eqnarray}}
On expanding $T(z_0^{(1)})$ along $x^*$ and taking norm of the equation $(\ref{eqn:327})$, we get
{\tiny
\begin{eqnarray}
\|z_0^{(2)}-x^*\|&\le&\frac{1}{1-w_0(\|z_0^{(1)}-x^*\|)}\int_0^1\psi((1-t)\|z_0^{(1)}-x^*\|)dt.\|z_0^{(1)}-x^*\|+\frac{[w(\|y_0-x^*\|)+w(\|z_0^{(1)}-x^*\|])}{1-w_0(\|z_0^{(1)}-x^*\|)}\nonumber\\
&\times&\|[T'(x^*)]^{-1}T(z_0^{(1)})\|+\frac{1}{1-w_0(\|x_0-x^*\|)}[w(\|y_0-x^*\|)+w(\|x_0-x^*\|)]\nonumber\\
&\times&\|[T'(x^*)]^{-1}T(z_0^{(1)})\|\nonumber\\
&\le&\frac{1}{1-w_0(\|z_0^{(1)}-x^*\|)}\int_0^1\psi((1-t)\|z_0^{(1)}-x^*\|)dt.\|z_0^{(1)}-x^*\|+\frac{[w(\|y_0-x^*\|)+w(\|z_0^{(1)}-x^*\|])}{1-w_0(\|z_0^{(1)}-x^*\|)}\nonumber\\
&\times&\frac{1}{1-p(\|x_0-x^*\|)}\bigg(\int_0^1{\bigg(}\psi_0(t\|z_0^{(1)}-x^*\|){+1\bigg)}dt\bigg)\|z_0^{(1)}-x^*\|+\frac{[w(\|y_0-x^*\|)+w(\|x_0-x^*\|)]}{1-w_0(\|x_0-x^*\|)}\nonumber\\
&\times&\frac{1}{1-p(\|x_0-x^*\|)}\bigg(\int_0^1{\bigg(}\psi_0(t\|z_0^{(1)}-x^*\|){+1\bigg)}dt\bigg)\|z_0^{(1)}-x^*\|.
\end{eqnarray}}
Thus, we have
\begin{eqnarray}
\|z_0^{(2)}-x^*\|\le \eta_3(\|x_0-x^*\|)\|x_0-x^*\|<\rho.
\end{eqnarray}
Now, from the last sub-step of the method $(\ref{eqn:13})$, we have
{\small
\begin{eqnarray}\label{eqn:329}
&&x_1-x^*=z_0^{(2)}-x^*-(2[T'(y_0)]^{-1}-[T'(x_0)]^{-1})T(z_0^{(2)})\nonumber\\
&&=(z_0^{(2)}-x^*-[T'(z_0^{(2)})]^{-1}T'(z_0^{(2)}))+[T'(z_0^{(2)})]^{(-1)}T'(x^*)T'(x^*)^{-1}[T'(y_0)-T'(z_0^{(1)})]\nonumber\\
&&.[T'(y_0)]^{(-1)}T'(x^*)[T'(x^*)]^{-1}T(z_0^{(2)})+[T'(x_0)]^{(-1)}T'(x^*)[T'(x^*)]^{-1}[T'(y_0)-T'(x_0)]\nonumber\\
&&.[T'(y_0)]^{(-1)}T'(x^*)[T'(x^*)]^{-1}T(z_0^{(2)}).
\end{eqnarray}}
On expanding $T(z_0^{(2)})$ along $x^*$ and taking norm of the equation $(\ref{eqn:329})$, we get
{\tiny
\begin{eqnarray}
\|x_1-x^*\|&\le&\frac{1}{1-w_0(\|z_0^{(2)}-x^*\|)}\int_0^1\psi((1-t)\|z_0^{(2)}-x^*\|)dt.\|z_0^{(2)}-x^*\|+\frac{[w(\|y_0-x^*\|)+w(\|z_0^{(2)}-x^*\|])}{1-w_0(\|z_0^{(2)}-x^*\|)}\nonumber\\
&\times&\|[T'(x^*)]^{-1}T(z_0^{(2)})\|+\frac{1}{1-w_0(\|x_0-x^*\|)}[w(\|y_0-x^*\|)+w(\|x_0-x^*\|)]\nonumber\\
&\times&\|[T'(x^*)]^{-1}T(z_0^{(2)})\|\nonumber\\
&\le&\frac{1}{1-w_0(\|z_0^{(2)}-x^*\|)}\int_0^1\psi((1-t)\|z_0^{(2)}-x^*\|)dt.\|z_0^{(2)}-x^*\|+\frac{[w(\|y_0-x^*\|)+w(\|z_0^{(2)}-x^*\|])}{1-w_0(\|z_0^{(2)}-x^*\|)}\nonumber\\
&\times&\frac{1}{1-p(\|x_0-x^*\|)}\bigg(\int_0^1{\bigg(}\psi_0(t\|z_0^{(2)}-x^*\|){+1\bigg)}dt\bigg)\|z_0^{(2)}-x^*\|+\frac{[w(\|y_0-x^*\|)+w(\|x_0-x^*\|)]}{1-w_0(\|x_0-x^*\|)}\nonumber\\
&\times&\frac{1}{1-p(\|x_0-x^*\|)}\bigg(\int_0^1{\bigg(}\psi_0(t\|z_0^{(2)}-x^*\|){+1\bigg)}dt\bigg)\|z_0^{(2)}-x^*\|.
\end{eqnarray}}
Thus, we have
\begin{eqnarray}
\|x_1-x^*\|\le \eta_4(\|x_0-x^*\|)\|x_0-x^*\|<\rho,
\end{eqnarray}
which shows that for $n=0,\ x_1\in B(x^*,\rho).$ By simply replacing $x_0, y_0, z_0^{(1)},z_0^{(2)}, x_1$ by $x_n, y_n, z_n^{(1)},z_n^{(2)}, x_{n+1}$ in the preceding estimates, we arrive at inequalities $(\ref{eqn:316})-(\ref{eqn:318})$. By the estimate, 
\begin{eqnarray}
\|x_{n+1}-x^*\|\le \eta_4(\|x_0-x^*\|)\|x_n-x^*\|<\rho.
\end{eqnarray}
We conclude that $\lim_{n\rightarrow\infty}x_n=x^*$ and $x_{n+1}\in B(x^*,r)$. Finally, to prove the uniqueness, let $y^*\in B(x^*,r)$ where $y^*\neq x^*$ with $T(y^*)=0$.
Define $F=\int_0^1T'(x^*+t(y^*-x^*))dt$. On expanding $T(y^*)$ along $x^*$ and using inequalities $(\ref{eqn:312})$ and $(\ref{eqn:319})$, we obtain
\begin{eqnarray}
\|[T'(x^*)]^{-1}\int_0^1[T'(x^*+t(y*-x^*)-T'(x^*)]dt\|\nonumber\\
\le\int_0^1\psi_0(t\|y^*-x^*\|)dt\le\int_0^1\psi_0(t\varrho)dt<1.
\end{eqnarray}
So, by Banach lemma, $\int_0^1[T'(x^*)]^{-1}[T'(x^*+t(y^*-x^*))]dt$  exists and invertible leading to the conclusion $x^*=y^*$, which completes the uniqueness part of the proof.
\end{proof}




\section{\bf Numerical example}
\textbf {Example 3.1}\cite{Behl1} Returning back to the motivational example given in the introduction of this study.
Let us consider now a nonlinear integral equation of Hammerstein type. These equations have strong physical background and arise in electro-magnetic fluid dynamics \cite{Poly}. This equation has the following form 
\begin{eqnarray*}
x(s)=u(s)+\int_{a}^bG(s,t)H(x(t))dt,\ \ a\le x\le b,
\end{eqnarray*}
for $x(s), u(s)\in C[a,b]$ with $-\infty<a<b<\infty$; $G$ is the Green function and $H$ is a polynomial function. The standard procedure to solve these type of equations contains in rewriting it as a nonlinear operator in a Banach space i.e. $F(x)=0,\ F:\psi \subseteq C[a,b]\rightarrow C[a,b]$ with a non-empty open convex subset and 
\begin{eqnarray*}
F[x(s)]=x(s)-u(s)-\int_{a}^bG(s,t)H(x(t))dt.
\end{eqnarray*}
considering the uniform norm $\|v\|=\max\limits_{s\in [a,b]}|v(s)|$. It was observed that in some cases boundedness conditions may not be satisfied since $F''(x)$ or $F'''(x)$ can be unbounded in a general domain. Thus, an alternative is looking a domain that contains the solution. But it is more convenient to apply the local convergence results obtained in our study in order to give the radius of convergence ball. 
Let $x^*=0$, then on using $(\ref{eqn:311})-(\ref{eqn:313})$ we have that, $w_0(t)=7.5t<w(t)=15t.$ It is straightforward to say  on the basis of the table $(\ref{tab:1})$ that the method $(\ref{eqn:13})$ has a larger domain of convergence in contrast to method (MMB) whose local convergence is analyzed by Behl et al. \cite{Behl1}. 
\begin{table}
\centering
\caption{Comparison of Convergence radius (Example 1)}\label{tab:1}
\begin{tabular}{|c| c| c| c| c| c| } \hline
$Radius$     &$\rho_1$  &$\rho_2$   &$\rho_3$  &$\rho_4$  &$\rho$\\ \hline
Method $(\ref{eqn:13})$        &0.0296296      &0.0205601    &0.0175449  & 0.0166341  & 0.0166341\\ \hline
MMB    &0.0666667  &0.0292298  &0.0118907 &0.00440901  &0.00440901  \\ 
\hline
\end{tabular}
\end{table}
\\

\textbf{Example 3.2}\cite{George}
Suppose that the motion of an object in three dimensions is governed by system of differential equations
\begin{eqnarray*}
f'_1(x)-f_1(x)-1=0,\\
f'_2(x)-(e-1)y-1=0,\\
f'_3(z)-1=0,
\end{eqnarray*}
with $x,y,z\in D=\overline{U(0,1)}$ for $f_1(0)=f_2(0)=f_3(0)=0.$ Then, the solution of the system is given for $v=(x,y,z)^T$ by function $F:=(f_1,f_2,f_3):D\rightarrow \mathbb{R}^3$ defined by
\begin{eqnarray*}
F(v)=(e^x-1, \frac{e-1}{2}y^2+y,z)^T.
\end{eqnarray*}
Then the first Fr\'echet derivative is given by
\begin{equation}
F'(v)=
\begin{pmatrix}
e^x && 0 && 0\\
0 && (e-1)y+1 && 0\\
0 && 0 && 1
\end{pmatrix}.
\end{equation}
Notice that $x^*=(0,0,0), \ F'(x^*)=F'(x^*)^{-1}=diag\{1,1,1\}$. Then, on using $(\ref{eqn:311})-(\ref{eqn:313})$ and on assuming $\psi_0(t)=\psi_0t,\ \psi(t)=\psi t$, we have that $\psi_0=e-1<\psi=e^{\frac{1}{\psi_0}}.$ The radius $\rho$ of convergence is computed in the Table $(\ref{tab:2})$.  Since the method $(\ref{eqn:13})$ has a larger radius of convergence as compared to the method (MMB),
 this means that method $(\ref{eqn:13})$ has a wider domain for the choice of the starting points. As a result, the approach being evaluated is more powerful.\\
\begin{table}
\centering
\caption{Comparison of Convergence radius (Example 2)}\label{tab:2}
\begin{tabular}{|c| c| c| c| c| c|} \hline
$Radius$   &$\rho_1$  &$\rho_2$  &$\rho_3$ &$\rho_4$  &$\rho$ \\ \hline
Method $(\ref{eqn:13})$    &0.164331  &0.135757  &0.119283   & 0.114151  & 0.114151 \\ \hline 
MMB    &0.382692  &0.198328  &0.0949498 &0.040525  &0.040525  \\ 
\hline
\end{tabular}
\end{table}
 \\           

\textbf{Example 3.3}\cite{Regmi} Let $X=Y=\psi=\mathbb{R}$. Define $F(x)= \sin x$. Using our assumptions, we get $F'(x)=\cos x$. Moreover, for $x^*=0$, on using $(\ref{eqn:311})-(\ref{eqn:313})$ and on assuming $\psi_0(t)=\psi_0t,\ \psi(t)=\psi t$,  it is derived that $\psi_0=\psi=1$.  It is clear to say on the basis of the Table $(\ref{tab:3})$ that the method $(\ref{eqn:13})$ has a larger radius of convergence as compared to the method (MMR) mentioned  in the reference \cite{Regmi}. So, we can conclude that  the presented method enlarges the radius of convergence ball.
\begin{table}
\centering
\caption{Comparison of Convergence radius (Example 3)}\label{tab:3}
\begin{tabular}{|c| c| c| c| c| c| } \hline
$Radius$ &$\rho_1$ & $\rho_2$ &$\rho_3$ &$\rho_4$ &$\rho$  \\ \hline
Method $(\ref{eqn:13})$         &0.285714   &0.238655    &0.210099    & 0.201186     & 0.201186\\ \hline
MMR    &0.44444  &0.277466  &0.15771 &  &0.15771  \\ 
\hline 
\end{tabular}
\end{table}
\\

\textbf{Example 3.4}\cite{Regmi} Consider the function $f$ defined on $D=[-\frac{1}{2},\frac{5}{2}]$ by
\begin{eqnarray*}\label{eqn:ex1}
f(x)=
\begin{cases}
      x^3log(x^2)+x^5-x^4, & \text{if \ $x\neq0$}\\
      0, & \text{if \ $x=0$}. 
\end{cases}
\end{eqnarray*}
The unique solution is $x^*=1$. The consecutive derivatives of $f$ are
\begin{eqnarray*}
f'(x)&=&3x^2logx^2+5x^4-4x^3+2x^2,\\
f''(x)&=&6xlogx^2+20x^3-12x^2+10x,\\
f'''(x)&=&6logx^2+60x^2-24x+22.
\end{eqnarray*}
It can be easily visible that $f'''$ is unbounded on $D$. Nevertheless, all the assumptions of the Theorem $(\ref{thm:31})$ for the iterative method $(\ref{eqn:13})$ are satisfied and hence applying the convergence results with $x^*=1$, we obtain $\psi_0(t)=\psi_0t,\ \psi(t)=\psi t$, it can be calculated that $\psi_0=\psi=96.6628$. Table $(\ref{tab:4})$ displays the radius $\rho$ of convergence by the discussed method $(\ref{eqn:13})$ along with the existing multi-step scheme (MMB). 
We discovered that when compared the provided method enlarges the radius of the convergence ball as compared to the existing one.
\begin{table}
\centering
\caption{Comparison of convergence radius (Example 4)}\label{tab:4}
\begin{tabular}{|c| c| c| c| c| c|} \hline
$Radius$  &$\rho_1$  &$\rho_2$  &$\rho_3$  &$\rho_4$  &$\rho$ \\ \hline
Method $(\ref{eqn:13})$        & 0.00295578    &  0.00246894    &0.00217353  &0.00208131   & 0.00208131  \\ \hline
MMB    &0.00689682  &0.00344841  &0.0015606 &0.000621105  &0.000621105  \\ 
\hline
\end{tabular}
\end{table}


\section{\bf Conclusions}
In the study of the convergence for the iterative methods, the major issues are the radius of convergence, the selection of the initial point and the uniqueness of the solution. Therefore, we have addressed these issues using efficient seventh order method by considering the sufficient convergence conditions that are weaker than Lipschitz and H$\ddot{o}$lder ones. That means, our analysis is applicable to solve such nonlinear problems when both Lipschitz and H$\ddot{o}$lder continuity conditions fail without applying higher-order derivatives. A convergence theorem for existence and uniqueness of the solution has been established followed by its error bounds. Consequently, we are extending the applicability of the method by solving some nonlinear equations employing our analytical results.

Akanksha Saxena\\
Department of Mathematics\\
Maulana Azad National Institute of Technology\\
 Bhopal, M.P. India-462003.\\
Email: akanksha.sai121@gmail.com.\\\\\\
J. P. Jaiswal\\
Department of Mathematics\\
Guru Ghasidas Vishwavidyalaya ( A Central University)\\
Bilaspur, C.G. India-495009.\\
Email: asstprofjpmanit@gmail.com.\\\\
K. R. Pardasani\\
Department of Mathematics\\
Maulana Azad National Institute of Technology\\
 Bhopal, M.P. India-462003.\\
Email: kamalrajp@rediffmail.com.\\\\
\end{document}